\documentclass[12pt]{amsart}

\usepackage{graphicx}
\usepackage{subfigure}
\usepackage{eepic}
\usepackage{xcolor}
\usepackage{tikz}
\usepackage{amsmath}
\usepackage{a4wide}
\usepackage[utf8]{inputenc}
\usepackage{amssymb}
\usepackage{amsopn}
\usepackage{epsfig}
\usepackage{amsfonts}
\usepackage{latexsym}
\usepackage{amsthm}
\usepackage{enumerate}
\usepackage[UKenglish]{babel}
\usepackage{verbatim}
\usepackage{color}

\newtheorem{theorem}{Theorem}[section]
\newtheorem{lemma}[theorem]{Lemma}
\newtheorem{proposition}[theorem]{Proposition}

\theoremstyle{definition}

\newtheorem{example}[theorem]{Example}

\theoremstyle{remark}

\numberwithin{equation}{section}

\newcommand{\R}{\ensuremath{\mathbb{R}}}

\renewcommand{\a}{\mathbf{a}}
\renewcommand{\b}{\mathbf{b}}

\newcommand{\w}{\mathbf{w}}
\newcommand{\W}{\mathbf{W}}

\newcommand{\D}{\mathcal{D}}
\newcommand{\set}[1]{\left\{#1\right\}}

\newcommand{\la}{\lambda}

\newcommand{\p}{\mathbf{p}}
\newcommand{\e}{\mathbf{e}}
\newcommand{\x}{\mathbf{x}}

\newcommand{\y}{\mathbf{y}}

\renewcommand{\L}{\mathcal{L}}

\begin{document}

\title[On the complexity of the set of codings]{On the complexity of the set of codings for self-similar sets and a variation on the construction of Champernowne.}

\author{Simon Baker}
\address{Mathematics institute, University of Warwick, Coventry, CV4 7AL, UK}
\email{simonbaker412@gmail.com}
\author{Derong Kong}
\address{College of Mathematics and Statistics, Chongqing University, 401331, Chongqing, P.R.China}
\email{{derongkong@126.com}}

\date{\today}

\subjclass[2010]{Primary 28A80; Secondary 11K16, 11K55}

\begin{abstract}
Let $F=\{\p_0,\ldots,\p_n\}$ be a collection of points in $\mathbb{R}^d.$ The set $F$ naturally gives rise to a family of iterated function systems consisting of contractions of the form $$S_i(\x)=\lambda \x +(1-\lambda)\p_i,$$ 
 where $\lambda \in(0,1)$. Given $F$ and $\lambda$ it is well known that there exists a unique non-empty compact set $X$ satisfying $X=\cup_{i=0}^n S_i(X)$. For each $\x \in X$ there exists a sequence $\a\in\{0,\ldots,n\}^{\mathbb{N}}$ satisfying $$\x=\lim_{j\to\infty}(S_{a_1}\circ \cdots \circ S_{a_j})(\mathbf{0}).$$ We call such a sequence a coding of $\x$. In this paper we prove that for any $F$ and $k \in\mathbb{N},$ there exists $\delta_k(F)>0$ such that if $\lambda\in(1-\delta_k(F),1),$ then every point in the interior of $X$ has a coding which is $k$-simply normal. Similarly, we prove that there exists $\delta_{uni}(F)>0$ such that if $\lambda\in(1-\delta_{uni}(F),1),$ then every point in the interior of $X$ has a coding containing all finite words. For some specific choices of $F$ we obtain lower bounds for $\delta_k(F)$ and $\delta_{uni}(F)$. We also prove some weaker statements that hold in the more general setting when the similarities in our iterated function systems exhibit different rates of contraction.  Our proofs rely on a variation of a well known construction of a normal number due to Champernowne, and an approach introduced by Erd\H{o}s and Komornik. 
 

\end{abstract}

\keywords{Expansions in non-integer bases, Digit frequencies.}
\maketitle

\section{Introduction}\label{sec:1}

A map $S:\mathbb{R}^d\to \mathbb{R}^d$ is called a contracting similitude if there exists $\lambda\in(0,1)$ such that $|S(\x)-S(\y)|=\la |\x-\y|$ for all $\x,\y\in \mathbb{R}^d$. We call a finite set of contracting similitudes an iterated function system or IFS for short. A well known result due to Hutchinson \cite{Hut} states that given an IFS $\Phi:=\{S_i\}_{i=0}^n,$ then there exists a unique non-empty compact set $X\subset \R^d$ satisfying $$X=\bigcup_{i=0}^n S_i(X).$$ 
We call $X$ the \emph{self-similar set} generated by $\Phi$. Many of the most well known examples of fractal sets are self-similar sets. For example the middle third Cantor set and the von Koch snowflake can be realised as self-similar sets for appropriate choices of iterated function systems (see \cite{Fal}).

When the images $\{S_i(X)\}_{i=0}^n$ are disjoint or have controlled overlaps, much is known about the properties of the attractor $X$ (see \cite{Fal}). Much less is known when the images $\{S_i(X)\}_{i=0}^n$ overlap significantly. One of the most important problems in Fractal Geometry is to describe the properties of $X,$ and measures supported on $X,$ when the $\{S_i(X)\}_{i=0}^n$ overlap significantly (see \cite{Hochman,Hochman2} and the references therein). To make progress with this problem it is often convenient to view $X$ as the image of a sequence space under a particular projection map. To avoid cumbersome notation, in what follows we will regularly adopt the convention: 
\[\D:=\{0,\ldots,n\}, \quad\D^*:=\bigcup_{j=0}^{\infty}\D^j,\quad \textrm{and}\quad \D^{\mathbb{N}}:=\{0,\ldots,n\}^{\mathbb{N}},\] where $\D^0$ consists of the empty word. We typically use $\a,$ $\mathbf{b}$ to denote an element of $\D^*$ or $\D^{\mathbb{N}}$. When we want to emphasise the digits appearing in $\a$ we use $(a_j)_{j=1}^{\infty}.$ Let $\pi:\D^{\mathbb{N}}\to X$ be  defined as follows: $$\pi(\a):=\lim_{j\to\infty} (S_{a_1}\circ \cdots \circ S_{a_j})(\mathbf{0}).$$ $\pi$ is the aforementioned projection map. Equipping $\D^{\mathbb{N}}$ with the product topology it can be shown that $\pi$ is continuous and surjective. Given $\x \in X$ we call a sequence $\a\in\D^{\mathbb{N}}$ a coding of $\x$ if $\pi(\a)=\x.$ In what follows we let $$\Sigma_{\Phi}(\x):=\{\a\in\D^{\mathbb{N}}:\pi(\a)=\x\}.$$ When the elements of the set $\{S_i(X)\}_{i\in \D}$ are well separated, then typically an $\x\in X$ will have a unique coding and so the set $\Sigma_{\Phi}(\x)$ does not exhibit any interesting behaviour. However, when the images $\{S_i(X)\}_{i\in \D}$ overlap significantly it can be the case that for a typical $\x$ the set of codings will be a large and complicated set. It is possible for $\Sigma_{\Phi}(\x)$ to be uncountable and even have positive Hausdorff dimension when $\D^{\mathbb{N}}$ is equipped with some reasonable metric (see \cite{BakG,Bak4,Bak3,Sid2,Sid3}). As a heuristic, it is reasonable to say that the more an IFS overlaps the larger the set $\Sigma_{\Phi}(\x)$ will be for a typical $\x,$ and vice-versa. As such the set of codings are important in the study of self-similar sets because their size provides a quantitative description of how an IFS overlaps. For more on this phenomenon and some analysis where this heuristic correspondence is made precise, we refer the reader to \cite{DK,FengHu,FS,Fur,Kem,Kem2}. These papers also demonstrate the important role the set of codings plays in the study of self-similar measures.

In this paper we study the combinatorial properties of the set of codings. We are motivated by the following general question. Suppose we are interested in a particular property of sequences in $\D^{\mathbb{N}},$ if our IFS overlaps sufficiently, does it guarantee that for a typical $\x\in X$ there will exist $\a\in \Sigma_{\Phi}(\x)$ satisfying this property? An affirmative answer to this question seems reasonable, since by the above heuristic, the more an IFS overlaps the larger we should expect $\Sigma_{\Phi}(\x)$ to be, and so we should expect a greater variety of sequences to appear within $\Sigma_{\Phi}(\x)$. Versions of this question were studied previously in \cite{BakD,BakE,BakKong,DKK,Gun,HS}. In \cite{Gun} some interesting connections were made between this problem and problems arising from analogue to digital conversion with background noise. In this paper we focus on the following two properties which measure the complexity of sequences.

Given $\b\in\D^k$ and $\a\in \D^{\mathbb{N}},$ we define the $\mathbf{b}$-frequency of $\a$ to be $$\textrm{freq}_{\b}(\a):=\lim_{m\to\infty}\frac{\#\{1\leq j \leq m:a_j\cdots a_{j+k-1}=\b\}}{m},$$ whenever the limit exists. Given $k\in \mathbb{N}$ we say that $\a$ is \emph{$k$-simply normal} if $\textrm{freq}_{\mathbf{b}}(\a)=(n+1)^{-k}$ for all $\mathbf{b}\in\D^k$. Essentially a sequence $\a$ is $k$-simply normal if each word of length $k$ occurs within $\a$ with the same likelihood. We emphasise at this point that $\D:=\{0,\ldots,n\}$ and so consists of $n+1$ digits. In this paper we study the following set: $$X_k:=\{\x\in X: \Sigma_{\Phi}(\x) \textrm{ contains a } k\textrm{-simply normal sequence}\}.$$ Another notion which describes the complexity of a sequence is that of universality. We call a sequence $\a\in \D^{\mathbb{N}}$ \emph{universal} if each element of $\D^{*}$ appears in $\a$, i.e., $\a$ contains all finite words. We will also study the set 
$$X_{uni}:=\{\x\in X:\Sigma_{\Phi}(\x) \textrm{ contains a universal sequence}\}.$$ Universal codings were originally introduced by Erd\H{o}s and Komornik in \cite{EK} in the setting of expansions in non-integer bases. For codings of self-similar sets they were studied by the first author in \cite{Bak4}.
 
The topic of digit frequencies and the complexity of codings is classical. It has strong connections with Ergodic Theory, Fractal Geometry, and Transcendental Number Theory. It has its origins in the pioneering work of Borel \cite{Bor} and Eggleston \cite{Egg}. For some more recent contributions on this topic we refer the reader to \cite{AB,BCH,HocShm} and the references therein. What distinguishes this work from much of what has appeared previously is the fact we are working in a setting where an $\x$ may have many codings.

In this paper we study the sets $X_k$ and $X_{uni}$ for the following parameterised families of IFSs. Given a set  $F:=\{\p_i\}_{i\in \D}$ consisting of vectors in $\mathbb{R}^d$, one can define a family of IFSs by defining for each $i\in D$ the similitude
\begin{equation}
\label{similarity}
S_i(\x)=\lambda_i \x +(1-\lambda_i)\p_i,
\end{equation} where for each $i\in \D$ we have $\lambda_i\in(0,1)$. Hiding the dependence upon $F$ and the contraction ratios we let $\Phi=\{S_i\}_{i\in \D}$ denote the IFS generated by these similitudes. In what follows, unless specified, we will always assume that $\Phi$ is an IFS consisting of similarities of the form given by \eqref{similarity}. We will also always have the underlying assumption that 
$F$ is not contained in a  $(d-1)$-dimensional affine subspace of $\mathbb{R}^d.$ 
If $F$ was contained in such a subspace then we could project to a lower dimensional Euclidean space where such a condition held. As such there is no loss of generality. If there exists $\lambda\in(0,1)$ such that $\lambda_i=\lambda$ for all $i\in \D$ we say that $\Phi$ is \emph{homogeneous.} We refer to the elements of $F$ as the fixed points of our IFS.

The following theorems are the main results of this paper. 

\begin{theorem}
	\label{k normal theorem}
	For any $F$ and $k\in\mathbb{N}$ there exists $\delta_{k}:=\delta_k(F)>0,$ such that if $\Phi$ is homogeneous and $\lambda\in(1-\delta_k,1),$ then $X_{k}=int(X)$.
	\end{theorem}
\begin{theorem}
	\label{universal theorem}
	For any $F$ there exists $\delta_{uni}:=\delta_{uni}(F)>0,$ such that if $\Phi$ is homogeneous and $\lambda\in(1-\delta_{uni},1),$ then $X_{uni}=int(X)$.
\end{theorem}
Here and hereafter we let $int(X)$ denote the interior of $X$. Theorems \ref{k normal theorem} and   \ref{universal theorem} are both existence results. In Section \ref{final discussion} we obtain explicit lower bounds for $\delta_k$ and $\delta_{uni}$ for certain classes of $F$. In particular when $d=1$ we obtain an explicit lower bound for $\delta_k$.

Earlier work on this topic appeared in \cite{BakD,BakE,BakKong}. In \cite{BakD,BakKong} we studied a family of homogeneous IFSs acting on $\mathbb{R}$ for which $X$ was an interval. Amongst other results we determined the optimal set of $\lambda$ for which we have $int(X)=X_1$. In \cite{BakE} the first author studied a more general family of homogeneous IFSs acting on $\mathbb{R}$. In this paper he showed that for any $\x\in int(X)$ the set of vectors $\{(\textrm{freq}_i(\a))_{i\in \D}: \a\in\Sigma_{\Phi}(\x)\}$ filled out the simplex of probability vectors on $n+1$ digits in a uniform way as $\lambda$ approached $1$.

At this point we contrast the arguments used in this paper with the arguments used in \cite{BakD,BakE,BakKong}. The arguments used in \cite{BakD,BakE,BakKong} made use of the obvious fact that if in a sequence $\a$ it is the case that $a_j=b$, then this does not impose any restrictions on the adjacent digits appearing within $\a$. This made controlling the quantity $\#\{1\leq j\leq m:a_j=b\}$ reasonably straightforward for certain codings that were constructed. Such a property does not hold for longer blocks. If $a_j\cdots a_{j+k-1}=\b$ for $k\geq 2$, then this clearly imposes some restrictions on what blocks of length $k$ can occur nearby. Consequently, the methods of \cite{BakD,BakE,BakKong} do not allow us to construct codings over which we have sufficient control over the quantity $\#\{1\leq j \leq m:a_j\cdots a_{j+k-1}=\b\}.$ In \cite{BakD,BakE,BakKong} we also made use of some dynamical arguments. These arguments were particularly effective because for the IFSs we were studying the corresponding self-similar set was an interval,  and so the geometry in this case was particularly simple. Working in an arbitrary Euclidean space we no longer have the same dynamical tools. To prove Theorem \ref{k normal theorem} and Theorem \ref{universal theorem} we will make use of a more combinatorial approach.

The rest of the paper is arranged as follows. In Section \ref{Sec2} we establish some notation and prove several technical results. In particular we generalise a construction of Champernowne to construct a large structured subset of $\D^{\mathbb{N}}$ consisting of $k$-simply normal sequences. The second half of Section \ref{Sec2} is concerned with deriving conditions for guaranteeing that the self-similar set of our IFS $\Phi$ is the convex hull of its fixed points. In Section \ref{sec3} we apply the results of Section \ref{Sec2} to prove various results of the form: if the contraction ratios appearing in $\Phi$ are sufficiently close to $1$, then $X_k$ is an open dense subset of $X$ of full Lebesgue measure. Some of the results of Section \ref{sec3} apply without the assumption $\Phi$ is homogeneous. In Section \ref{sec4} we generalise an argument of Erd\H{o}s and Komornik \cite{EK} to prove that $X_{uni}=int(X)$ when $\Phi$ is homogeneous and consists of $d+1$ maps with contraction ratios sufficiently close to $1$. In Section \ref{proofs} we use this result to prove Theorem \ref{universal theorem}. Theorem \ref{k normal theorem} will then follow as a corollary of Theorem \ref{universal theorem} and the results of Section \ref{sec3}. In Section \ref{final discussion} we give general conditions under which one can derive lower bounds for $\delta_k$ and $\delta_{uni}$. We apply this result to the study of expansions in non-integer bases. We also pose some open questions.

\section{Notation and preliminaries}
\label{Sec2}
\subsection{Notation}
Given a finite word $\mathbf{a}:=(a_j)_{j=1}^m$ let $S_{\mathbf{a}}:=S_{a_1}\circ \cdots \circ S_{a_m},$ let $|\a|$ denote the length of $\a$, and let $\mathbf{a}^{\infty}\in\D^{\mathbb{N}}$ denote the infinite concatenation of $\mathbf{a}$ with itself. For $\a,\b\in\D^{\mathbb{N}}$ we write $\a\prec \b$ if $\a$ is lexicographically strictly less than $\b$. Recall that  $\a$ is strictly less than $\b$ with respect to the lexicographic ordering if $a_1<b_1,$ or if there exists $l\in\mathbb{N}$ such that $a_j=b_j$ for all $1\leq j\leq l$ and $a_{l+1}<b_{l+1}$. We can extend the lexicographic ordering to elements of $\D^*$ by writing $\a\prec \b$ if $\a 0^{\infty}\prec \b 0^{\infty}$. Given $\mathbf{a}\in\D^*$ such that $\a\neq n^{|\a|},$ we let $\mathbf{a}^+$ be the lexicographically smallest word of length $|\a|$ that is strictly larger than $\a.$  Similarly, if $\a\neq 0^{|\a|}$ we let $\a^-$ be the lexicographically largest word of length $|\a|$ that is strictly smaller than $\a.$

\subsection{Preliminaries}

\subsubsection{A variation on the construction of Champernowne} A sequence $\a$ is called \emph{normal} if $\a$ is $k$-simply normal for all $k\in \mathbb{N}$. A construction of Champernowne \cite{Cha} gave the first explicit example of a normal sequence in $\{0,\ldots,9\}^{\mathbb{N}}$. More specifically, he proved that the sequence obtained by listing all the natural numbers in increasing order is normal, i.e., $$0\,1\,2\,3\,4\,5\,6\,7\,8\,9\,10\,11\,12\,13\,14\ldots
$$ was normal. Inspired by Champernowne's approach, in this section we devise a method for constructing a large structured set of $k$-normal sequences in $\D^\mathbb N$.

Let us denote the elements of $\D^k$ written in increasing lexicographic order by $\{\w_l\}_{l=0}^{{(n+1)^k}-1}.$ So $\w_{m}\prec\w_{m'}$ whenever $m<m'$. For the purpose of exposition we state here some terms in $\{\w_l\}_{l=0}^{{(n+1)^k}-1}:$
$$\w_0=0^k, \w_1=0^{k-1}1, \ldots, \w_{n}=0^{k-1}n, \w_{n+1}=0^{k-2}10, \ldots, \w_{(n+1)^k - 1}=n^k.$$ We make use of the notation $\w_l:=w_{1,l}\cdots w_{k,l}$. Using the $\{\w_l\}$ we now define the following collection of words:
\begin{align*}
\W_0&= \w_0 \w_1 \cdots \w_{(n+1)^k -1}\\
\W_1&= \w_1 \w_2 \cdots \w_{(n+1)^k - 1}\w_0\\
\W_2&=\w_2 \w_3 \cdots \w_{(n+1)^k -1}\w_0\w_1\\
\cdots&\\
\cdots&\\
\W_{n}&=\w_{n} \w_{n+1}\cdots \w_{(n+1)^k -1 }\w_0\cdots \w_{n-1}.
\end{align*}  
We emphasise here that each $\W_i$ has length $k\cdot (n+1)^k$ and begins with $0^{k-1}.$  For example, when $n=1$ and $k=2$ we have $$\W_0= 00011011 \textrm{ and } \W_1=01101100.$$

\begin{lemma}
	\label{missing zeros}
To any $\W_i$ associate the word $\mathbf{c}:=\W_i 0^{k-1}$. For any $\w_l\in \D^k$ we then have $$\#\{1\leq j \leq k\cdot (n+1)^k:c_j\cdots c_{j+k-1}=\w_l\}= k.$$
\end{lemma}
\begin{proof}
In what follows $\w_l\in\D^k$ is fixed. We remark that any $1\leq j \leq k \cdot (n+1)^k$ can be uniquely expressed as $j=m\cdot k +r$ for some $0\leq m <(n+1)^k$ and $1\leq r \leq k$. As such to prove our result it suffices to show that for each $1\leq r \leq k$ there exists a unique $0\leq m <(n+1)^k$ such that the corresponding $j=m\cdot k +r$ satisfies $c_j\cdots c_{j+k-1}=\w_l$. This will be our strategy of proof. It is convenient to split our argument into the following two cases.
\\

\noindent \textbf{Case 1. $\W_i=\W_0$.}  When $r=1$ it is immediate from the definition of $\W_0$ that the unique $j=m\cdot k+1$ such that $c_j\cdots c_{j+k-1}=\w_l$ is when $j=l\cdot k +1$. Now let us fix $r>1.$ We introduce the notation $\w_{pre,r}:=w_{1,l}\cdots w_{k+1-r,l}$ for the first $k+1-r$ digits of $\w_l,$ and   $\w_{suf,r}:=w_{k+2-r,l}\cdots w_{k,l}$ for the last $r-1$ digits of $\w_l$. There are three subcases to consider.
\begin{itemize}
	\item Suppose $\w_{pre,r}\neq n^{k+1-r}$ so $\w_{pre,r}^+$ is well defined. Using the fact that $\W_0$ is all of the elements of $\D^k$ written in increasing order, we can deduce that there exists a unique $\w_p$ such that 
	\begin{equation}
	\label{1a}
	\w_p\w_{p+1}=\w_{suf,r}\w_{pre,r}\w_{suf,r}\w_{pre,r}^+=\w_{suf,r}\w_{l}\w_{pre,r}^+.
	\end{equation}
	\item Suppose $\w_{pre,r}= n^{k+1-r}$ and $\w_{suf,r}\neq 0^{r-1}$ so $\w_{suf,r}^{-}$ is well defined. Using the fact that $\W_0$ is all of the elements of $\D^k$ written in increasing order, we can deduce that there exists a unique $\w_p$ such that \begin{equation}
	\label{1b}
	\w_p\w_{p+1}=\w_{suf,r}^{-}\w_{pre,r}\w_{suf,r}0^{k+1-r}=\w_{suf,r}^{-}\w_{l} 0^{k+1-r}.
	\end{equation}
	\item Suppose $\w_{pre,r}=n^{k+1-r}$ and $\w_{suf,r}=0^{r-1}.$ Then the only position where these words can occur in succession is at the end of $\mathbf{c}$ where we have
	\begin{equation}
	\label{1c}\w_{(n+1)^k-1}0^{k-1}=n^{r-1}\w_{pre,r} \w_{suf,r} 0^{k-r}=n^{r-1}\w_{l} 0^{k-r}. 
	\end{equation}
\end{itemize} Equations \eqref{1a}, \eqref{1b}, and \eqref{1c} uniquely determine our value of $m$ for each of these three subcases. This completes our proof for the case $\W_i=\W_0$. 
\\

\noindent \textbf{Case 2. $\W_i\neq \W_0$.}  As in the case where $\W_i=\W_0,$ when $r=1$ there is obviously a unique $j=m\cdot k+1$ such that $c_j\cdots c_{j+k-1}=\w_l.$ Now let us fix $r>1$. We see from the construction of $\mathbf{c}$ that a block $\w_l$ is followed by the next lexicographically largest block $\w_{l+1}$ unless $\w_l=\w_{(n+1)^k -1}=n^k$ or $\w_l=\w_{i-1}$. We will use this fact implicitly in our deductions below. We now proceed via a case analysis. There are five subcases to consider.
\begin{itemize}
	\item Suppose $\w_{pre,r}\neq n^{k+1-r}$ and $\w_{pre,r}\neq 0^{k-r}(i-1).$ Then there exists a unique $p$ such that $\w_p \w_{p+1}$ appears as two successive block in $\mathbf{c}$ and 
	\begin{equation}
	\label{2a}
	\w_{p}\w_{p+1}=\w_{suf,r}\w_{pre,r}\w_{suf,r}\w_{pre,r}^+=\w_{suf,r}\w_{l}\w_{pre,r}^+.
	\end{equation}
	\item Suppose $\w_{pre,r}=n^{k+1-r}$ and $\w_{suf,r}\neq 0^{r-1}.$ Then there exists a unique $p$ such that $\w_p \w_{p+1}$ appears as two successive blocks in $\mathbf{c}$ and  
	\begin{equation}
	\label{2b}
	\w_p\w_{p+1}=\w_{suf,r}^{-}\w_{pre,r}\w_{suf,r}0^{k+1-r}=\w_{suf,r}^{-}\w_{l}0^{k+1-r}. 
	\end{equation}
	\item Suppose $\w_{pre,r}=n^{k+1-r}$ and $\w_{suf,r}= 0^{r-1}$. Then the only position where $\w_l$ can occur is when
	\begin{equation}
	\label{2c}\w_{(n+1)^k -1}\w_0=n^{r-1}\w_{pre,r}\w_{suf,r}0^{k+1-r}=n^{r-1}\w_{l}0^{k+1-r}. 
	\end{equation}
	\item Suppose $\w_{pre,r}= 0^{k-r}(i-1)$ and $\w_{suf,r}\neq 0^{r-1}$. Then there exists a unique $p$ such that $\w_p \w_{p+1}$ appears as two successive blocks in $\mathbf{c}$ and  \begin{equation}
	\label{2d}
	\w_p\w_{p+1}=\w_{suf,r}\w_{pre,r}\w_{suf,r}\w_{pre,r}^+=\w_{suf,r}\w_{l}\w_{pre,r}^+. 
	\end{equation}  
	\item Suppose $\w_{pre,r}= 0^{k-r}(i-1)$ and $\w_{suf,r}=0^{r-1}.$ Then the only position where these words can occur in succession is at the end of $\mathbf{c}$ where we have 
	\begin{equation}
	\label{2e}
	\w_{i-1}0^{k-1}=0^{r-1}0^{k-r}(i-1)0^{r-1}0^{k+1-r}=0^{r-1}\w_{l}0^{k+1-r}.
	\end{equation} 
	\end{itemize}
Equations \eqref{2a}, \eqref{2b}, \eqref{2c}, \eqref{2d}, and \eqref{2e} uniquely determine our value for $m$ in each of the five subcases. This completes our proof when $\W_i\neq \W_0$. 
\end{proof}

\begin{proposition}
	\label{k normal prop}
Every element of $\{\W_0,\ldots, \W_{n}\}^{\mathbb{N}}$ is $k$-simply normal.	
\end{proposition}
\begin{proof}
 Let $\a\in \{\W_0,\ldots, \W_n\}^{\mathbb{N}}$ and $\w_l\in \D^k$ be arbitrary. Note that each $\W_i$ begins with $0^{k-1}$. Therefore by an application of Lemma \ref{missing zeros} we have $$\#\{1\leq j \leq k\cdot (n+1)^k:a_j\cdots a_{j+k-1}=\w_l\}= k.$$ More generally, by Lemma \ref{missing zeros} we see that for any $p\in\mathbb{N}$ we have $$\#\{p k\cdot (n+1)^k +1\leq j \leq (p+1) k\cdot (n+1)^k:a_j\cdots a_{j+k-1}=\w_l\}= k.$$
 Therefore for any $p\in\mathbb{N}$ we have 
$$\#\{1\leq j \leq p\cdot k\cdot (n+1)^k:a_j\cdots a_{j+k-1}=\w_l\}=p\cdot k.$$  This implies $\textrm{freq}_{\w_l}(\a)=1/(n+1)^k$ as required. Since $\a$ and $\w_l$ were arbitrary our result follows.
\end{proof}

\subsubsection{Self-similar sets with no holes}
In many of our later proofs it will be important to be able to assert that the self-similar set $X$ of $\Phi=\set{\la_i\x+(1-\la_i)\p_i}_{i\in \D}$ equals the convex hull of its fixed points $F=\set{\p_i}_{i\in \D}$, i.e.,
\begin{equation}
\label{no holes}
X=\textrm{conv}(F).
\end{equation} Here and in what follows we use $\textrm{conv}(F)$ to denote the convex hull of a finite set of vectors $F\subseteq \mathbb{R}^d$. In this subsection we give sufficient conditions for \eqref{no holes} to hold. Much of our analysis is a generalisation of results appearing in \cite{BMS} and \cite{Sid3} to the case where our IFS contains similitudes with different rates of contraction. Lemma \ref{no holes lemma 2} also provides a more succinct proof of Proposition $2.4$ from \cite{Sid3}. 

\begin{lemma}
\label{no holes lemma}
Suppose $\Phi=\set{S_i}_{i=0}^d$ consists of $d+1$ maps and $F$ is not contained in a $(d-1)$-dimensional affine subspace. If the contraction ratios satisfy $\sum_{i=0}^d\lambda_i \geq d$, then $X=\textrm{conv}(F).$
\end{lemma}
\begin{proof}
By performing a change of coordinates we may assume that $F=\{\p_i\}_{i=0}^d$ where $\p_0=(0,\ldots,0)$ and $\p_i$ is the $i$-th vector in the standard unit basis of $\mathbb{R}^d$ for $1\leq i\leq d$. For these vectors it is straightforward to check that 
$$\textrm{conv}(F) =\Big\{\x=(x_1,\ldots, x_d)\in\mathbb{R}^d: x_j\geq 0\, , \sum_{j=1}^d x_j\leq 1\Big\}.$$ Let $\Delta$ denote the right hand side of the above equation. It is a simple exercise to check that 
$$S_0(\Delta):=\Big\{\x\in\mathbb{R}^d: x_j\geq 0,\,  \sum_{j=1}^d x_j\leq \lambda_0\Big\},$$ and for $1\leq i \leq d$
$$S_i(\Delta):= \Big\{\x\in\mathbb{R}^d: x_j\geq 0,\, x_i\geq  1-\lambda_i,\, \sum_{j=1}^d x_j\leq 1\Big\}.$$
Recall that $X$ is the self-similar set generated by $\Phi=\set{S_i}_{i=0}^d$. If $X\neq \Delta$ then $\Delta\neq\cup_{i=0}^d S_i(\Delta).$ Since $S_i(\Delta)\subseteq \Delta$ for each $i$,  there must exists $\x \in \Delta$ satisfying 
\[\sum_{j=1}^d x_j> \lambda_0,\quad \textrm{and}\quad  x_i< 1-\lambda_i\quad \textrm{for each }1\leq i \leq d.\]
Substituting the second inequality into the first we see that if such an $\x$ exists,  then we must have $d>\sum_{i=0}^d\lambda_i.$ Therefore if $d
 \leq\sum_{i=0}^d\lambda_i$, no such $\x$ can exist. So $\Delta=\cup_{i=0}^d S_i(\Delta)$ and $\Delta$ is the self-similar set for $\Phi$. 
\end{proof}

\begin{lemma}
	\label{no holes lemma 2}
Suppose $\Phi=\set{S_i}_{i\in\D}$ is such that 
\begin{equation}
\label{contraction sum}
\min_{\substack{A\subseteq \D\\ \#A=d+1}}\sum_{i\in A}\lambda_i\geq d.
\end{equation} 
Then $X=\textrm{conv}(F).$
\end{lemma}

\begin{proof}

Let us proceed via induction on the dimension $d$. Let $d=1$ and without loss of generality assume $\p_0=\min_{i\in\D}\{\p_i\}$ and $\p_n=\max_{i\in \D}\{\p_i\}.$ We have $S_0([\p_0,\p_n])=[\p_0,\p_0+\lambda_0(\p_n-\p_0)]$ and $S_n([\p_0,\p_n])=[\p_n-\lambda_n(\p_n-\p_0),\p_n].$ By our assumption we know that $\lambda_0+\lambda_n\geq 1$. It follows that $\p_n-\lambda_n(\p_n-\p_0)\leq \p_0+\lambda_0(\p_n-\p_0)$ and so $[\p_0,\p_n]=S_0([\p_0,\p_n])\cup S_n([\p_0,\p_n]).$ Since $S_i([\p_0,\p_n])\subseteq [\p_0,\p_n]$ for all the remaining $i$ we see that $X=[\p_0,\p_n]=\textrm{conv}(F).$

Let us assume the result is true for $d=m$. We now show that the lemma holds when $d=m+1$. To prove our inductive step we make use of a well known theorem of Caratheodory which states that if $F$ is a finite set of points in $\mathbb{R}^{m+1}$, then any point in $\textrm{conv}(F)$ can be expressed as the convex combination of $m+2$ points from $F$ (see \cite{Roc}).

Applying Caratheodory's theorem in $\mathbb{R}^{m+1}$ we have 
$$\textrm{conv}(F)=\bigcup_ {\substack{B\subseteq  F\\ \# B={m+2}}}\textrm{conv}(B).$$
 Since $X\subseteq \textrm{conv}(F)$ it suffices to show that $\textrm{conv}(B)\subseteq X$ for each $B\subseteq F$ consisting of $m+2$ elements. If the elements of $B$ are not contained in a $m$-dimensional affine subspace,  then we can apply Lemma \ref{no holes lemma} to assert that $\textrm{conv}(B)= X_{B}$, where $X_{B}$ is the self-similar set determined by the IFS $\{S_i:\p_i \in B\}$. Since $X_{B}\subseteq X$,  we have the desired inclusion when the elements of $B$ are not contained in a $m$-dimensional affine subspace. If the elements of $B$ are contained in 
such a subspace, we can identify this subspace with $\mathbb{R}^m$, we can then apply our inductive hypothesis when $d=m$ to the IFS determined by $\{S_i:\p_i \in B\}$ acting upon $\mathbb{R}^{m}$. To apply our inductive hypothesis when $d=m$ it only remains to check that 
$$\min_{\substack{A\subseteq \{i:\p_i\in B\}\\ \# A=m+1}}\sum_{i\in A}\lambda_i \geq m.$$ However this holds because we are assuming \eqref{contraction sum} holds when $d=m+1$ and $\lambda_i\in(0,1)$ for all $i\in\D$.
\end{proof}

It follows from the construction of the $\W_i$ that each digit in $\D$ occurs in $\W_i$ exactly $k\cdot(n+1)^{k-1}$ times. Therefore the contraction ratio $\la_{\W_i}$ of each $S_{\W_i}$ is independent of $i$ and equals $\prod_{i\in \D}\lambda_i^{k\cdot(n+1)^{k-1}}.$ Given $k\geq 1$ we let $$\mathcal{P}_k:=\Big\{(\lambda_i)_{i\in \D}:\lambda_i\in(0,1) \textrm{ and } \prod_{i\in \D}\lambda_i^{k\cdot(n+1)^{k-1}}\geq \frac{d}{d+1}\, \Big\}.$$ 
Therefore, $\mathcal{P}_k$ is precisely the set of $(\lambda_i)_{i\in \D}$ such that 
\[
\min_{\substack{A\subseteq \D\\ \#A=d+1}}\sum_{i\in A}\la_{\W_i}\ge d.
\]
So Lemma \ref{no holes lemma 2} can be applied to the IFS $\{S_{\W_i}\}_{i\in\mathcal{D}}.$ In what follows we denote the IFS determined by $\{S_{\W_i}\}_{i\in\mathcal{D}}$ by $\Phi_k,$ and the corresponding self-similar set by $X_{\Phi_k}$.

We now combine Proposition \ref{k normal prop} and Lemma \ref{no holes lemma 2} to give sufficient conditions guaranteeing that $X_k$ contains a metrically and topologically large subset of $X$.  

\begin{proposition}
	\label{generic k}
Let $k\geq 1$ and suppose $(\lambda_i)_{i\in \D}\in\mathcal{P}_k.$ If the set of points $\{\pi(\W_{i}^{\infty})\}_{i\in \D}$ is not contained in a $(d-1)$-dimensional affine subspace of $\mathbb{R}^d,$ then $X_k$ contains an open dense subset of $X$. Moreover, Lebesgue almost every $\x\in X$ is contained in $X_k$. 
\end{proposition}

\begin{proof}
Each $S_{\W_i}$ can be written as $$S_{\W_i}(\x)=\prod_{i\in \D}\lambda_i^{k\cdot(n+1)^{k-1}}\cdot\x +\Big(1-\prod_{i\in \D}\lambda_i^{k\cdot(n+1)^{k-1}}\Big)\cdot\pi(\W_{i}^{\infty}).$$ So each $S_{\W_i}$ can be written in the form appearing in \eqref{similarity} where the appropriate fixed point is $\pi(\W_{i}^{\infty}).$ It follows from Lemma \ref{no holes lemma 2} that if  $(\lambda_i)_{i\in \D}\in\mathcal{P}_k$ and the fixed points  $\{\pi(\W_{i}^{\infty})\}_{i\in \D}$ are not contained in a $(d-1)$-dimensional affine subspace of $\mathbb{R}^d$, then $X_{\Phi_k}=\textrm{conv}(\{\pi(\W_{i}^{\infty})\}_{i\in \D})$ and has non-empty interior. Importantly, by Proposition \ref{k normal prop} each element of $X_{\Phi_k}$ has a $k$-normal coding, i.e., $X_{\Phi_k}\subseteq X_k$. 

Consider the set $$X_{pre,\Phi_k}:= \bigcup_{\a\in \D^*}S_{\a}(int(X_{\Phi_k})).$$  Since $X_{\Phi_k}$ has non-empty interior,  it follows that $X_{pre,\Phi_k}$ is an open dense subset of $X$. Moreover, each $\x\in X_{pre,\Phi_k}$ has a coding of the form $\a\b$ where $\a$ is a finite word and $\b$ is a $k$-normal sequence. Since whether a sequence is $k$-normal is independent of an initial block, it follows that every element of $X_{pre,\Phi_k}$ has a $k$-normal coding and therefore $X_k$ contains an open dense subset of $X$.

It remains to prove that $\L(X\setminus X_k)=0.$ Here $\L$ denotes the $d$-dimensional Lebesgue measure. Fix $\x\in X$ and let $\a$ be a coding of $\x$. For any $r>0$ sufficiently small there exists $n\in\mathbb{N}$ such that 
\begin{equation}
\label{density equation}
\lambda_{a_1}\cdots \lambda_{a_n} Diam(X)<r\leq  \lambda_{a_1}\cdots \lambda_{a_{n-1}} Diam(X).
\end{equation} 
Since $\x\in S_{a_{1}\cdots a_n}(X)$,  it follows that $S_{a_{1}\cdots a_n}(X)\subseteq B(\x,r).$ Using \eqref{density equation} it follows that
\begin{align*}
\L(B(\x,r)\setminus X_k)&\leq \L(B(\x,r)\setminus S_{a_{1}\cdots a_n}(X_{\Phi_k}))\\
&= \L(B(0,1))\cdot r^d - (\lambda_{a_1}\cdots \lambda_{a_n})^d \L(X_{\Phi_k})\\
&\leq \L(B(0,1))\cdot r^d- \Big(\frac{ \min_{i\in \D} \lambda_i}{Diam(X)}\Big)^d \L(X_{\Phi_k})r^d\\
&=\L(B(0,1))r^d\Big(1-\Big(\frac{ \min_{i\in \D} \lambda_i}{Diam(X)}\Big)^d \frac{\L(X_{\Phi_k})}{\L(B(0,1))}\Big).
\end{align*}Therefore for all $\x\in X\setminus X_k$ we have $$\limsup_{r\to 0} \frac{\L(B(\x,r)\cap (X\setminus X_k))}{\L(B(\x,r))}<1.$$ Applying the Lebesgue density theorem we may conclude that $\L(X\setminus X_k)=0.$ 
\end{proof}
Proposition \ref{generic k} gives conditions guaranteeing that a typical element of $X$, in the sense of both topology and measure, will be contained in $X_k$. This topological statement should be contrasted with the folklore result that for self-similar sets satisfying the strong separation condition, the set of $\x$ whose unique coding is not $1$-normal contains a dense $G_{\delta}$ set and so is topologically generic. It is also worth commenting on our proof of the measure counterpart of Proposition \ref{generic k}. Typically one would prove a result of this type in one of two ways. One could define a continuous map $T:X\to X$ and study the ergodic $T$-invariant measures. If one of these measures were equivalent to the Lebesgue measure restricted to $X$ then one could hope that $T$ would yield some information about the set of codings for a Lebesgue generic $\x$. Alternatively, one could consider a measure supported on $\D^\mathbb{N}$ and hope that it projects under $\pi$ to a measure which is equivalent to the Lebesgue measure restricted to $X$. Knowledge about the measure supported on $\D^{\mathbb{N}}$ can then be transferred into knowledge about the set of codings for a Lebesgue generic $\x$. Our proof of Proposition \ref{generic k} didn't make use of either of these methods. Our proof instead relied upon constructing a sizeable set of points in $X_k$ and then using the fact that $S_{\a}(X_k)\subseteq X_k$ for all $\a\in \D^*$. The reason we can adopt such an approach is because our IFS contains such significant overlaps.

The problem with Proposition \ref{generic k} is verifying when the set $\{\pi(\W_i^{\infty})\}_{i\in \D}$ is not contained in a $(d-1)$-dimensional affine subspace of $\mathbb{R}^d.$ We concern ourselves with this verification in the next section.

\section{Metric and topological properties of $X_k$}
\label{sec3}
In this section we prove several results which follow from Proposition \ref{generic k}. The proofs of each of these statements rely upon showing that $\{\pi(\W_i^{\infty})\}_{i\in \D}$ is not contained in a $(d-1)$-dimensional affine subspace of $\mathbb{R}^d$ for some appropriate subset of the space of contractions.

For the purposes of exposition in what follows we let $$\mathcal{M}_{k}:=\Big[\Big(\frac{d}{d+1}\Big)^{\frac{1}{k\cdot (n+1)^k}},1\Big).$$ $\mathcal{M}_k$ is simply the set of $\lambda\in(0,1)$ such that $(\lambda,\ldots,\lambda)\in \mathcal{P}_k$.

\begin{proposition}
	\label{prop1}
Let $k\geq 1$ and suppose $\Phi$ is homogeneous. Then for all but at most finitely many $\lambda\in \mathcal{M}_k,$ the set $X_k$ contains an open dense subset of $X$ and Lebesgue almost every $\x\in X$ is contained in $X_k$. In particular, there exists $\delta_{k}':=\delta_{k}'(F)>0$ such that if $\lambda\in(1-\delta_{k}',1),$ then $X_k$ contains an open dense subset of $X$ and Lebesgue almost every $\x\in X$ is contained in $X_k$.
\end{proposition}
\begin{proof}
	 By our underlying assumptions we know that $F=\{\p_0,\ldots,\p_n\}$ is not contained in an $(d-1)$-dimensional affine subspace of $\mathbb{R}^d.$ As such we may assume without loss of generality that $\p_0=\mathbf{0}$ and $\p_1,\ldots, \p_d$ are linearly independent.
	
	 By Proposition \ref{generic k} to prove our result it suffices to show that for all but at most finitely many values of $\lambda \in \mathcal{M}_k$ the set $\{\pi(\W_i^{\infty})\}_{i\in\D}$ is not contained in a $(d-1)$-dimensional affine subspace. Consider the set of fixed points $\{\pi(\W_0^{\infty}),\pi(\W_1^{\infty}),\ldots, \pi(\W_{d}^{\infty})\}.$ To prove $\{\pi(\W_i^{\infty})\}_{i=0}^d$ is not contained in a $(d-1)$-dimensional affine subspace it suffices to show that  the vectors $\{\pi(\W_1^{\infty})-\pi(\W_0^{\infty}),\ldots, \pi(\W_{d}^{\infty})-\pi(\W_0^{\infty})\}$ are linearly independent. Consider the matrix whose rows are made up of these vectors:
	$$M(\lambda):=\begin{pmatrix} 
	\pi(\W_1^{\infty})-\pi(\W_0^{\infty})  \\
	\cdots \\
	\cdots \\
	\pi(\W_{d}^{\infty})-\pi(\W_0^{\infty})
	\end{pmatrix}$$ Consider the function $P(\lambda):=Det(M(\lambda))$. The vectors $\{\pi(\W_1^{\infty})-\pi(\W_0^{\infty}),\ldots, \pi(\W_{d}^{\infty})-\pi(\W_0^{\infty})\}$ are linearly independent if and only if $P(\lambda)\neq 0$. It therefore suffices to show that $P(\lambda)\neq 0$ for all but at most finitely many values of $\lambda \in \mathcal{M}_k$. 
	
	For each $i\in \D$ the vector $\pi(\W_i^{\infty})$ consists of $d$ entries each taking the form $p(\lambda)/q(\lambda)$ for two polynomials $p,q\in\mathbb{R}[x].$ This follows since each entry within $\pi(\W_i^{\infty})$ can be expressed as a geometric series in $\lambda$. Alternatively, one could see this as a consequence of the fact that $\pi(\W_i^{\infty})$ is the unique fixed point of $S_{\W_i}$. It follows from the definition of the determinant that $P(\lambda)=f(\lambda)/g(\lambda)$ for some $f,g\in\mathbb{R}[x].$ Importantly $P(\lambda)=0$ if and only if $f(\lambda)=0$. The polynomial $f$ either has finitely many roots or is the constant function zero. We now show that $f$ is not the constant zero function. 
	
Recall from the definition that $\W_i$ begins with $0^{k-1}i.$ Since we've assumed $\p_0=\mathbf{0}$ it follows from the definition of the coding map $\pi$ that 
$$\lambda^{-(k-1)}M(\lambda)=\begin{pmatrix} 
\pi(1\a^1)-\pi(0\a^0)  \\
\cdots \\
\cdots \\
\pi(d\a^{d})-\pi(0\a^0)
\end{pmatrix}$$ for some infinite sequences $\a^0,\ldots,\a^{d}\in \D^{\mathbb{N}}$. It follows from the definition of $\pi$ that as $\lambda \to 0$ we have $\pi(i\a^{i})\to \p_i,$ for each $i\in \D$. Therefore $$\lambda^{-(k-1)}M(\lambda)\to \begin{pmatrix} 
\p_1 \\
\cdots \\
\cdots \\
\p_d
\end{pmatrix}$$ as $\lambda \to 0$. Since the vectors $\p_1,\ldots,\p_d$ are linearly independent, it follows that 
\[Det(\lambda^{-(k-1)}M(\lambda))=\lambda^{-(k-1)d}P(\lambda)\neq 0\] for all $\lambda$ sufficiently close to $0$. Therefore $f(\lambda)$ is not the constant zero polynomial, and so $P(\la)$ has finitely many roots. This completes our proof.
\end{proof}
Note that Proposition \ref{prop1} is a weak version of Theorem \ref{k normal theorem}. To prove the full theorem we will need Theorem \ref{universal theorem}.

The following theorem applies when our contraction ratios aren't equal. 
\begin{theorem}
	Let $k\geq 1$. Within $\mathcal{P}_k$ there exists an open dense set $\mathcal{O}$ such that for any $(\lambda_0,\ldots,\lambda_n)\in \mathcal{O}$ the set $X_k$ contains an open dense subset of $X$ and Lebesgue almost every $\x\in X$ is contained in $X_k$.
\end{theorem}
\begin{proof}
	As in the proof of Proposition \ref{prop1} we may assume $\p_0=\mathbf{0}$ and the vectors $\p_1,\ldots,\p_d$ are linearly independent. Let $$P(\lambda_0,\ldots,\lambda_n):=\begin{vmatrix} 
	\pi(\W_1^{\infty})-\pi(\W_0^{\infty})  \\
	\cdots \\
	\cdots \\
	\pi(\W_{d}^{\infty})-\pi(\W_0^{\infty})
	\end{vmatrix}.$$  By Proposition \ref{generic k} and similar arguments to those used in the proof of Proposition \ref{prop1}, it suffices to show that the set of $(\lambda_0,\ldots,\lambda_n)$ such that $P(\lambda_0,\ldots,\lambda_n)\neq 0$ is an open dense subset of $\mathcal{P}_k$. By continuity the set of $(\lambda_0,\ldots,\lambda_n)\in \mathcal{P}_k$ such that $P(\lambda_0,\ldots,\lambda_n)\neq 0$ is an open set.  It remains to show the density part of our result. Fix $(\lambda_0,\ldots,\lambda_n)\in \mathcal{P}_k$ and let $\epsilon>0$ be arbitrary. There exists an interval $I\subset (0,1)$ and integers $\{k_i\}_{i\in \D}$ such that for any $\lambda\in I$ we have 
	\[\lambda^{k_i}\in(\lambda_i-\epsilon,\lambda_i+\epsilon)\quad \textrm{for each }i\in \D.\] Replicating the argument given in the proof of Proposition \ref{prop1}, it can be shown that $P(\lambda^{k_0},\ldots,\lambda^{k_n})=f(\lambda)/g(\lambda)$ for some $f,g\in \mathbb{R}[x]$,  where $f$ is not the constant zero polynomial. Therefore $P(\lambda^{k_0},\ldots,\lambda^{k_n})$ has finitely many zeros and we can find $\lambda_*\in I$ such that $P(\lambda_{*}^{k_0},\ldots,\lambda_{*}^{k_n})\neq 0.$ Since $\lambda_{*}^{k_i}\in(\lambda_i-\epsilon,\lambda_i+\epsilon)$ for each $i\in \D$ and $\epsilon$ is arbitrary, our result follows.
\end{proof}

\begin{theorem}
	\label{d=1 theorem}
	Assume $d=1$ and $k\geq 1$. Then for any $(\lambda_0,\ldots,\lambda_n)\in\mathcal{P}_k$ the set $X_k$ contains an open dense subset of $X$ and Lebesgue almost every $\x\in X$ is contained in $X_k$.
\end{theorem}
\begin{proof}
Verifying $\{\pi(\W_i^{\infty})\}_{i\in \D}$ is not contained in a $(d-1)$-dimensional affine subspace is much more straightforward when $d=1$. We simply have to prove that there exists $i,j\in\mathcal{D}$ such that $\pi(\W_i^{\infty})\neq \pi(\W_j^{\infty}).$ We may assume without loss of generality that $\p_0=0,$ $\p_i\geq 0$ for all $i\in\mathcal{D}$, and there exists $i\in \D$ such that $\p_i>0$. Since there exists $\p_i>0$ it follows that $\pi(\W_0^{\infty})>0.$ It then follows from the construction of $\W_0$ and $\W_1$ that $\lambda_{0}^{-k}(\pi(\W_{0}^{\infty}))=\pi(\W_{1}^{\infty})$.  Since $\lambda_{0}\in(0,1)$ we must have $\pi(\W_{0}^{\infty})< \pi(\W_{1}^{\infty})$. By Proposition \ref{generic k} our result follows.
\end{proof}

\section{Universal codings}
\label{sec4}
 Universal codings were originally introduced by Erd\H{o}s and Komornik in \cite{EK}. They were interested in $q$-expansions of real numbers. These are defined as follows. Given $q\in (1,n+1]$ we say that $\a\in \{0,\ldots,n\}^{\mathbb{N}}$ is a $q$-expansion of $x$ if $$x=\sum_{j=1}^{\infty}\frac{a_j}{q^j}.$$ An $x$ has a $q$-expansion if and only if $x\in[0,\frac{n}{q-1}]$. Expansions of this type exhibit very different behaviour to the well known binary, ternary, decimal expansions. We refer the reader to the survey \cite{Ko} for more on these expansions. When studying $q$-expansions one naturally ends up studying the IFS $\{\frac{x+i}{q}\}_{i=0}^n$. A sequence $\a$ is a coding of $x$ with respect to this IFS if and only if $\a$ is a $q$-expansion of $x$. As such for this IFS we have $X=[0,\frac{n}{q-1}]$.
 
 Erd\H{o}s and Komornik gave necessary conditions for guaranteeing that every $x\in(0,\frac{n}{q-1})$ has a universal coding (see Theorem \ref{EK} below). To prove this result they studied the following parameterised family of sets. To each $q\in(1,n+1]$ let
$$Z_{n}(q):=\Big\{\sum_{j=0}^{m}a_jq^j: m\in\mathbb{N}, a_j\in\{0,\ldots,n\}\Big\}.$$ 
Since $Z_n(q)$ is a discrete set,  it can be written as $\{y_{l,n}(q)\}_{l=1}^{\infty}$ where $y_{l,n}(q)<y_{l+1,n}(q)$ for all $l\in\mathbb{N}$. To study the distribution of $Z_{n}(q)$ within $\mathbb{R}$ it is natural to consider the quantities:
\begin{align*}
l_n(q)&:=\liminf_{l\to\infty}\left(y_{l+1,n}(q)-y_{l,n}(q)\right)\\
L_n(q)&:=\limsup_{l\to\infty}\left(y_{l+1,n}(q)-y_{l,n}(q)\right).
\end{align*}Much has been written on the quantities $l_n(q)$ and $L_n(q),$ see \cite{AkiKom, EK, Feng, Ko, SidSol} and the references therein. In \cite{EK} it was shown that $l_n(q)>0$ whenever $q$ is a Pisot number. Recall that a Pisot number is an algebraic integer whose Galois conjugates all have modulus strictly less than one. This result gave rise to the conjecture that $l_n(q)>0$ if and only if $q$ is a Pisot number. This conjecture was shown to be true in a recent paper by Feng \cite{Feng}, who built upon previous work of Akiyama and Komornik \cite{AkiKom}. Feng's result also has the following useful implication for the quantity $L_n(q)$.
\begin{theorem}[Feng \cite{Feng}]
	\label{Feng thm}
	If $q\in(1,\sqrt{n+1})$ and $q^2$ is not a Pisot number,  then $L_n(q)=0$. In particular, if $q\in(1,\sqrt{1.3247\ldots})$,  then $L_n(q)=0$.
\end{theorem} Note that $x'=1.3247\ldots$ is the smallest Pisot number. Its minimal polynomial is $x^3-x-1$.

The significance of the quantity $L_n(q)$ for us is demonstrated in the following result of Erd\H{o}s and Komornik. 
\begin{theorem}[Erd\H{o}s and Komornik \cite{EK}]
	\label{EK}If $L_n(q)=0$, then $X_{uni}=(0,\frac{n}{q-1})$. 
\end{theorem} In this section we will always assume that $\Phi$ is homogeneous. Under this assumption it can be shown that the coding map $\pi$ takes the form 
\begin{equation}
\label{sum}
\pi(\a)=(1-\lambda)\sum_{j=1}^{\infty}\lambda^{j-1}\p_{a_j}.
\end{equation} In what follows we make use of the following family of expanding maps. To each $i\in \D$ let $$T_i(\x)=\frac{\x-(1-\lambda)\p_i}{\lambda}.$$ Note that $T_i$ is simply the inverse of $S_i$. Given $\a=(a_i)_{i=1}^j\in \D^{*}$ we let $T_\a$ denote the map $T_{\a_j}\circ \cdots \circ T_{\a_1}$. To each $x\in X$ we associate the set $$\Omega_{\Phi}(\x):=\{\a\in \D^{\mathbb{N}}:(T_{a_j}\circ \cdots \circ T_{a_1})(\x)\in X,\,  \forall j\in\mathbb{N}\}.$$ Adapting the arguments of \cite{BakG} the following lemma can be shown to hold.
\begin{lemma}
\label{bijection lemma}
$\Sigma_{\Phi}(\x)=\Omega_{\Phi}(\x)$
\end{lemma} Having the dynamical interpretation of a coding provided by Lemma \ref{bijection lemma} helps simplify certain arguments. The purpose of this section is to prove the following result on the size of $X_{uni}$ which holds for IFSs acting on $\mathbb{R}^d$ consisting of $d+1$ maps.

\begin{proposition}
	\label{d+1 prop}
	Suppose $\Phi=\set{S_i}_{i=0}^d$ consists of $d+1$ maps and the fixed points $\{\p_0,\ldots,\p_d\}$ are not contained in a $(d-1)$-dimensional affine subspace. If $\lambda\in(2^{-1/2d},1)$ and $\lambda^{-2d}$ is not a Pisot number, then $X_{uni}=int(X)$. In particular,  if $\lambda\in(1.3247^{-1/2d},1)$, then $X_{uni}=int(X).$
\end{proposition} Applying a change of coordinates as in the proof of Lemma \ref{no holes lemma}, we see that to prove Proposition \ref{d+1 prop} it suffices to consider the case where $\p_0$ is the $\mathbf{0}$ vector in $\mathbb{R}^d$, and each $\p_i$ is the $i$-th unit vector in the standard basis of $\mathbb{R}^d$. To emphasise when we are dealing with these vectors we denote them by $\mathbf{e}_0,\ldots, \e_d$. The following lemma is the first step towards proving Proposition \ref{d+1 prop}.

\begin{lemma}
	\label{greedy}
	Let $\mathbf{e}_0,\ldots, \e_d$ be the fixed points of $\Phi$ and  $\lambda\in [\frac{d}{d+1},1)$. If $\x\in int(X)$, then there exists $\a\in\D^{*}$ such that $T_{\a}(\x)\in (0,(1-\lambda)]^d.$
\end{lemma}
\begin{proof}
We start by remarking that by Lemma \ref{no holes lemma} we know that 
$$X=\Big\{\x\in\mathbb{R}^d: x_i\geq 0\, , \sum_{i=1}^d x_i\leq 1\Big\}.$$ To prove our lemma we devise an algorithm for constructing a coding. This algorithm is similar in spirit to the quasi-greedy algorithm from expansions in non-integer bases (see \cite{Ko}). 

We construct a coding in keeping with the following rules. Fix $\x\in X$.

\begin{enumerate}
	\item If $\x\in int(X)$ and there exists $i\neq 0$ such that $T_i(\x)\in int(X),$ apply one of these $T_i$. 
	\item If $\x\in int(X)$ and there exists no $i\neq 0$ such that $T_i(\x)\in int(X)$ apply $T_0$.
	\item If $\x\in \partial X$ choose $T_i$ arbitrarily so that $T_i(\x)\in X$. 
\end{enumerate} To check that repeatedly applying these rules yields an element of $\Omega_{\Phi}(\x),$ we have to check that for any $\x\in X$ our rules yield a map $T_i$ such that $T_i(\x)\in X$. For the first and third rule this is obviously true. It remains to check the second rule. If $\x\in int(X)$ is such that there exists no $i\neq 0$ such that $T_i(\x)\in int(X),$ then it can be shown that $\x\in (0,(1-\lambda)]^d$. Applying $T_0$ we obtain $T_0((0,(1-\lambda)]^d)=(0,(1-\lambda)/\lambda]^d.$ To see that $(0,(1-\lambda)/\lambda]^d \subseteq X$ it suffices to check $d\cdot \frac{1-\lambda}{\lambda}\leq 1$. However this follows from our assumption $\lambda\in[\frac{d}{d+1},1)$. Therefore the second rule yields a map satisfying $T_i(\x)\in X$, and our algorithm yields an element of $\Omega_{\Phi}(\x)$ for each $\x\in X$.

We remark here that our algorithm has the property that if we apply a map determined by rule $1$, then it has to be followed by a map determined by either rule $1$ or rule $2$. We also remark that we only apply rule $2$ when $\x\in  (0,(1-\lambda)]^d.$ 

Now we apply our algorithm to construct our desired sequence $\a\in \D^{*}$. If $\x\in (0,(1-\lambda)]^d$ then there is nothing to prove. Let us assume $\x\in int(X)\setminus (0,(1-\lambda)]^d.$ By our above remark we see that it suffices to show that we eventually apply a map corresponding to rule $2$, since the previous maps determined by our rules must have mapped $\x$ into $(0,(1-\lambda)]^d.$ Since $ \x\in int(X)\setminus (0,(1-\lambda)]^d$ we must first of all apply a map corresponding to rule $1$. Since a rule $1$ map must be followed by either a rule $1$ map or a rule $2$ map, it suffices to show that we cannot apply the maps generated by rule $1$ indefinitely. By construction a map corresponding to rule $1$ cannot equal $T_0$. Therefore if we were able to apply rule $1$ indefinitely, Lemma \ref{bijection lemma} would imply that $\x$ has a coding containing no zeros. It can be shown that any such $\x=(x_1,\ldots,x_d)$ must satisfy $\sum_{i=1}^d x_i=1,$ and therefore must be contained in the boundary of $X$. This contradicts our assumption $\x\in int(X)$. Therefore we must eventually apply a rule $2$ map and $\x$ must eventually be mapped into $(0,(1-\lambda)]^d.$
\end{proof}

For our purposes we need the following analogue of $Z_{n}(q)$:
$$Z^{d}(\lambda):=\Big\{(1-\lambda)\sum_{j=1}^m \e_{a_j}\lambda^{-j}:m\in\mathbb{N},\, (a_j)\in\set{0,1,\ldots,d}^{\mathbb{N}}\Big\}.$$ 
\begin{lemma}
	\label{half plane dense}
Suppose $\lambda\in(2^{-1/2d},1)$ and $\lambda^{-2d}$ is not a Pisot number. Then for any $\epsilon>0$ there exists $C>0$ such that $Z^d(\lambda)$ is $\epsilon$-dense in $\cap_{i=1}^d\{\x:x_i\geq C\}.$
\end{lemma}
\begin{proof}
Fix $\lambda\in(2^{-1/2d},1)$ such that $\lambda^{-2d}$ is not a Pisot number. For each $1\leq i \leq d$ let $$Z^d(i,\lambda):=\Big\{(1-\lambda)\sum_{\substack{1\leq j \leq m\\ j=i \mod d}} a_j\e_i\lambda^{-j}:m\in\mathbb{N},\, a_j\in\{0,1\}\Big\}.$$ Since the elements of $Z^d(i,\lambda)$ consist of sums of scaled copies of a single $\e_i$, the set $Z^d(i,\la)$ is a subset of the axis spanned by $\e_i$. Note that we have the inclusion 
\begin{equation}
\label{sum inclusion}
Z^d(1,\lambda)+Z^d(2,\lambda)+\cdots + Z^d(d,\lambda)\subseteq Z^d(\lambda).
\end{equation} Applying Theorem \ref{Feng thm} we know that for any $\epsilon>0$ there exists $C_1>0$ such that $\{\sum_{j=0}^m a_j \lambda^{-dj}:m\in \mathbb{N}, a_j\in\{0,1\}\}$ is $\epsilon$-dense in $[C_1,\infty).$ Importantly each $Z^d(i,\lambda)$ is simply a copy of $\{\sum_{j=0}^m a_j \lambda^{-dj}:m\in \mathbb{N}, a_j\in\{0,1\}\}$ that has been scaled by a power of $\lambda$ and then rotated to align with the $i$-axis. Therefore we may conclude that for any $\epsilon>0$ there exists $C>0$ such that $Z^d(i,\lambda)$ is $\epsilon$-dense in $\{\x:x_i\geq C,\, x_k=0 \textrm{ for } k\neq i\}$ for any $1\leq i\leq d$. Our result now follows from \eqref{sum inclusion}.
\end{proof}

Before moving on to our proof of Proposition \ref{d+1 prop} we make a simple observation. By Lemma \ref{no holes lemma 2} for $\lambda$ sufficiently close to $1$ we have $X=\textrm{conv}(F)$. Therefore if $\x$ is contained in the boundary of $X$ for $\lambda$ sufficiently close to $1,$ it must be contained in a bounding hyperplane of $\textrm{conv}(F)$ of dimension $d-1$. Call this hyperplane $V$. Since $F$ is not contained in any $(d-1)$-dimensional affine subspace, there must exist $\p_i\in F$ such that $\p_i\notin V.$ One can then show by a simple argument that since $\x\in V$ it cannot have a coding containing the digit $i$. As such we automatically have the inclusions 
\[X_k\subseteq int(X)\quad \textrm{and}\quad X_{uni}\subseteq int(X).\]
 Therefore to prove Proposition \ref{d+1 prop}, and later Theorems \ref{k normal theorem} and \ref{universal theorem}, it will be sufficient to show that the opposite inclusions holds for $\lambda$ sufficiently close to $1$. Equipped with this observation and the lemmas above we are now in a position to prove Proposition \ref{d+1 prop}.

\begin{proof}[Proof of Proposition \ref{d+1 prop}]
As previously remarked upon, by a change of coordinates we may assume without loss of generality that our fixed points are $\e_0,\ldots,\e_d$. Let us now fix $\lambda$ satisfying the hypothesis of our proposition. It can be shown that $2^{-1/2d}> \frac{d}{d+1}$ for all $d\geq 1,$ therefore by Lemma \ref{no holes lemma} we know that $X=\textrm{conv}(\{\e_0,\ldots,\e_d\})$. By the above remark it now suffices to show that $int(X)\subseteq X_{uni}$.

Since $2^{-1/2d}> \frac{d}{d+1}$ for all $d\geq 1$,  we can apply Lemma \ref{greedy}. As such for any $\x\in int(X)$ there exists $\a$ such that $T_{\a}(\x)\in (0,1-\lambda]^d$. Therefore, we see by Lemma \ref{bijection lemma} that there is no loss of generality in assuming to begin with that $\x\in (0,1-\lambda]^d.$ Let us now fix $\x\in (0,1-\lambda]^d$ and let $\mathbf{B}_1,\mathbf{B}_2,\ldots$ be an enumeration of all the elements of $\D^{*}=\set{0,1,\ldots,d}^*.$ 

Suppose $\mathbf{B}_1=b_1\ldots b_{k}$. Consider the vector $$\x\cdot \lambda^{-l}-(1-\lambda)\sum_{j=1}^{k}\e_{b_j}\lambda^{j-1}.$$ Since $\x\in (0,1-\lambda]^d$, we have that for any $C>0$ this vector is contained in $\cap_{i=1}^d\{\x:x_i\geq C\}$ for $l$ sufficiently large. Applying Lemma \ref{half plane dense} for an appropriate choice of $\epsilon,$ we see that for $l$ sufficiently large there exists $c_1\cdots c_p\in \D^*$ such that $c_p\neq 0$ and
\begin{equation}
\label{step1}
\x\cdot \lambda^{-l}-(1-\lambda)\sum_{j=1}^{k}\e_{b_j}\lambda^{j-1}\in (1-\lambda)\sum_{j=1}^{p}\e_{c_j}\lambda^{-j} + (0,(1-\lambda)\lambda^{k}]^d.
\end{equation}  Rewriting \eqref{step1} we obtain 
\begin{equation}
\label{step2}\x\in  (1-\lambda)\sum_{j=1}^{p}\e_{c_j}\lambda^{-j+l}+(1-\lambda)\sum_{j=1}^{k}\e_{b_j}\lambda^{j+l-1}+(0,(1-\lambda)\lambda^{k+l}]^d.
\end{equation}
The two summations appearing in \eqref{step2} share no common powers of $\lambda.$ What is more, since $\x\in (0,1-\lambda]^d,$ none of the coordinates of $\x \cdot \lambda^{-l}$ can exceed $(1-\lambda)\cdot \lambda^{-l}$. This implies that $p<l$.
Combining these two facts with \eqref{step2} we see that there exists $m_0=k+l$ and a word $\a_0=a_{1,0}\ldots a_{m_0,0}$ such that $\a_0$ contains $\mathbf{B}_1$ as a subword and   
$$\x\in(1-\lambda)\sum_{j=1}^{m_0}\e_{a_{j,0}}\lambda^{j-1}+(0,(1-\lambda)\lambda^{m_0}]^d.$$ Let $\x_1$ be such that $\x_1\in (0,1-\lambda]^d$ and 
\begin{equation}
\label{sub1}
\x=(1-\lambda)\sum_{j=1}^{m_0}\e_{a_{j,0}}\lambda^{j-1}+\x_1\cdot \lambda^{m_0}.
\end{equation}Replacing $\x$ with $\x_1$ and $\mathbf{B}_1$ with $\mathbf{B}_2$ we can repeat the argument above to show that there exists a word $\mathbf{d}_1\in \D^{*}$ such that $\mathbf{d}_1$ contains $\mathbf{B}_2$ as a subword and  \begin{equation}
\label{sub2}
\x_1\in(1-\lambda)\sum_{j=1}^{|\mathbf{d}_1|}\e_{d_j}\lambda^{j-1}+(0,(1-\lambda)\lambda^{|\mathbf{d}_1|}]^d.
\end{equation}
Let $\a_1:=\a_0\mathbf{d_1}=a_{1,1}\ldots a_{m_1,1}$ with $m_1=m_0+|\mathbf{d}_1|$. Then  $\a_1$ contains $\mathbf{B}_1$ and $\mathbf{B}_2$ as subwords. Substituting \eqref{sub2} into \eqref{sub1} we obtain 
$$\x\in(1-\lambda)\sum_{j=1}^{m_1}\e_{a_{j,1}}\lambda^{j-1}+(0,(1-\lambda)\lambda^{m_1}]^d.$$ 
 We can repeat this step indefinitely and show that for any $q\in\mathbb{N}$ there exists a sequence $\a_q=a_{1,q}\ldots a_{m_q,q}$ containing $\mathbf{B}_1,\ldots,\mathbf{B}_{q+1}$ as subwords and satisfying 
\begin{equation}
\label{laststep}
\x\in(1-\lambda)\sum_{j=1}^{m_q}\e_{a_{j,q}}\lambda^{j-1}+(0,(1-\lambda)\lambda^{m_q}]^d. 
\end{equation}It follows from our construction that for any $q_1<q_2$ the word $\a_{q_1}$ is a prefix of $\a_{q_2}.$ It follows that the infinite sequence $\a_{\infty}$ obtained as the component-wise limit of the $\a_q$ is well defined. Moreover $\a_{\infty}$ contains all finite blocks and by \eqref{laststep} satisfies $$\x=(1-\lambda)\sum_{j=1}^{\infty}\e_{a_{j,\infty}}\lambda^{j-1}.$$ Appealing to the formulation of a coding provided by \eqref{sum} we see that $\a_{\infty}$ satisfies the desired properties. 
\end{proof}
In the proof of Proposition \ref{d+1 prop} we've made no effort to optimise the quantities appearing in its statement. It is likely that one can improve upon these estimates.

\section{Proofs of Theorems \ref{k normal theorem} and \ref{universal theorem}}
\label{proofs}
In this section we prove Theorems \ref{k normal theorem} and \ref{universal theorem}. For our proofs it is useful to have the following lemma.

\begin{lemma}
	\label{dense}
If $\a\in\D^{\mathbb{N}}$ is a universal coding for $\x,$ then $\{T_{a_1\ldots a_j}(\x):j\geq 1\}$ is dense in $X.$ 
\end{lemma} The proof of Lemma \ref{dense} is straightforward and therefore omitted.

\begin{proposition}
	\label{universal prop}
	Suppose $\Phi$ is homogeneous,  $\lambda\in(2^{-1/2d},1)$ and $\lambda^{-2d}$ is not a Pisot number. If $\x\in int(\textrm{conv}(B))$ for some $B\subseteq F$ consisting of $d+1$ fixed points which are not contained in any $(d-1)$-dimensional affine subspace, then $\x\in X_{uni}$.
\end{proposition}

\begin{proof}
Let us start by fixing $d+1$ fixed points $B$ that are not contained in any $(d-1)$-dimensional affine subspace. Let $\mathbf{B}_1,\mathbf{B}_2,\ldots$ be an enumeration of the elements of $\D^*.$ We emphasise here that $\D$ is a potentially larger digit set than $\{i:\p_i\in B\}$. We now also fix $\textbf{A}\in \D^*$ such that $S_{\mathbf{A}}(X)\subseteq int(\textrm{conv}(B))$. It is useful to remark at this point that for any $\a\in \D^*$ the set $S_{\a}(X)$ has non-empty interior and  $$S_{\a}(X)=\{\x\in X:T_{\a}(\x)\in X\}.$$
 Let us now fix $\x\in int(\textrm{conv}(B))$. By Proposition \ref{d+1 prop} we know that $\x$ has a universal coding for the restricted digit set $\{i:\p_i\in B\}.$ Consider the set $S_{\mathbf{A}\mathbf{B}_1\mathbf{A}}(X)$. Since $S_{\mathbf{A}}(X)\subseteq int(\textrm{conv}(B))$, we also have $S_{\mathbf{A}\mathbf{B}_1\mathbf{A}}(X)\subseteq int(\textrm{conv}(B))$. Therefore, by Lemma \ref{dense}, there exists $\a\in \cup_{j=0}^{\infty}\{i:\p_i\in B\}^j$ such that $T_{\a}(\x)\in S_{\mathbf{A}\mathbf{B}_1\mathbf{A}}(X)$. Therefore $T_{\a \mathbf{A}\mathbf{B}_1}(\x)\in S_{\mathbf{A}}(X).$ By construction $S_{\mathbf{A}}(X)\subseteq int(\textrm{conv}(B)),$ therefore 
 \[T_{\a \mathbf{A}\mathbf{B}_1}(\x)\in int(\textrm{conv}(B)).\]
Note by Proposition \ref{d+1 prop} that $T_{\a \mathbf{A}\mathbf{B}_1}(\x)$ has a universal coding for the digit set $\{i:\p_i\in B\}$. As such there exists $\a_1\in \cup_{j=0}^{\infty}\{i:\p_i\in B\}^j$ such that $T_{\a \mathbf{A}\mathbf{B}_1\a_1}(\x)\in S_{\mathbf{A}\mathbf{B}_2\mathbf{A}}(X).$ Which by the above implies 
\[T_{\a \mathbf{A}\mathbf{B}_1\a_1 \mathbf{A}\mathbf{B}_2}(\x)\in int(\textrm{conv}(B)).\]
Therefore by Proposition \ref{d+1 prop} $T_{\a \mathbf{A}\mathbf{B}_1\a_1 \mathbf{A}\mathbf{B}_2}(\x)$ has a universal coding for the digit set $\{i:\p_i\in B\}$.
 
  Clearly one can repeat the above step indefinitely for successive $\mathbf{B}_k$'s. This yields an element of $\Omega_{\Phi}(\x)$ which contains every element of $\D^*$ as a subword. By Lemma \ref{bijection lemma} $\x$ has a universal coding for digit set $\D$. 
\end{proof}
We also require the following strengthening of Caratheodory's theorem. 

\begin{lemma}
\label{strong caratheodory}
Let $B\subset \mathbb{R}^d$ be a finite set of points not contained in any $(d-1)$-dimensional affine subspace. For any $\x\in \textrm{conv}(B),$ there exists $B'\subseteq B$ such that $B'$ consists of $d+1$ extremal points of $\textrm{conv}(B),$ $\x\in \textrm{conv}(B'),$ and $B'$ is not contained in any $(d-1)$-dimensional affine subspace.
\end{lemma}

\begin{proof}

Let $B_{ext}\subseteq B$ denote the set of extremal points of $\textrm{conv}(B).$ By the Krein-Milman theorem (see \cite{DS}) we have 
\begin{equation}
\label{extreme}
\textrm{conv}(B_{ext})=\textrm{conv}(B).
\end{equation}
Since $B$ is not contained in any $(d-1)$-dimensional affine subspace, we also have that $B_{ext}$ is not contained in any $(d-1)$-dimensional affine subspace. 

Let us recall here Caratheodory's theorem which states that if $B$ is a finite set of points in $\mathbb{R}^{d}$, then any point in $\textrm{conv}(B)$ can be expressed as the convex combination of $d+1$ points from $B$ (see \cite{Roc}). Combining Caratheodory's theorem applied to $B_{ext}$ with \eqref{extreme}, we see that for any $\x\in \textrm{conv}(B)$ there exists $B_1\subseteq B_{ext}$ such that $\# B_1=d+1$ and $\x\in \textrm{conv}(B_1)$. If the elements of $B_1$ are not contained in a $(d-1)$-dimensional affine subspace we are done.
 If not, then $\textrm{conv}(B_1)$ is contained in a $(d-1)$-dimensional affine subspace $V_1$ that is contained in $\mathbb{R}^d$. Identifying $V_1$ with $\mathbb{R}^{d-1}$ we can apply Caratheodory's theorem again to assert that there exists $B_2\subset B_1$ such that $\#B_2=d$ and $\x\in \textrm{conv}(B_2)$. If the elements of $B_2$ are not contained in a $(d-2)$-dimensional affine subspace of $V_1,$ then we pick $\p\in B_{ext}$ such that $\p\notin V_1$. In which case $B'=B_2\cup \{\p\}$ satisfies the desired properties. Such a $\p$ exists since $B_{ext}$ is not contained in a $(d-1)$-dimensional affine subspace. Suppose the alternative holds and $B_2$ is contained in a $(d-2)$-dimensional affine subspace of $V_1$ which we call $V_2$. Identifying $V_2$ with $\mathbb{R}^{d-2}$ and applying Caratheodory's theorem, we may assert that there exists $B_3\subset B_2$ such that $\#B_3=d-1$ and $\x\in\textrm{conv}(B_3)$.
 
Repeating the above steps we can conclude that eventually one of two outcomes occurs. Either there exists a set $B_*\subseteq B_{1}$ such that $\# B_*\geq 2,$ the elements of $B_*$ are not contained in a $(\# B_*-2)$-dimensional affine subspace and $\x\in \textrm{conv}(B_*),$ or alternatively $\x\in B_{ext}$. In the former case we may then choose $\p_1,\ldots, \p_{d+1-\# B_*}\in B_{ext}$ such that $B'=B_*\cup \{\p_1,\ldots, \p_{d+1-\#B_*}\}$ is not contained in a $(d-1)$-dimensional affine subspace. In the latter case we choose  $\p_1,\ldots, \p_{d}\in B_{ext}$ such that $B'=\set{\x}\cup \{\p_1,\ldots,\p_{d}\}$ is not contained in a $(d-1)$-dimensional affine subspace. The fact that these vectors exist follows because the elements of $B_{ext}$ are not contained in a $(d-1)$-dimensional affine subspace. In either case the constructed $B'$ has the desired properties.
\end{proof}
With Lemma \ref{dense}, Proposition \ref{universal prop}, and Lemma \ref{strong caratheodory} we can now prove Theorem \ref{universal theorem}. 

\begin{proof}[Proof of Theorem \ref{universal theorem}]
By the remarks preceding the proof of Proposition \ref{d+1 prop}, it suffices to show that $int(X)\subseteq X_{uni}$ for $\lambda$ sufficiently close to $1$. We prove this inclusion via induction on the dimension $d$ of the Euclidean space $\Phi$ is acting upon. Let us start with the case $d=1$.

Suppose $F \subset \mathbb{R}.$ Without loss of generality we may assume that $\p_0=\min F$ and $\p_n=\max F.$ Therefore $\p_0<\p_n$ and $X=\textrm{conv}(\{\p_0,\p_n\})$  for $\la$ sufficiently close to $1$. It follows from Theorem \ref{Feng thm}, Theorem \ref{EK} and a simple scaling argument, that if $\lambda\in(1.3247^{-1/2},1)$ then every $\x\in int(\textrm{conv}(\{\p_0,\p_n\}))$ has a universal coding for the digit set $\{0,n\}$. By Proposition \ref{universal prop} it follows that every $\x\in int(\textrm{conv}(\{\p_0,\p_n\}))=int(X)$ has a universal coding for our original digit set $\D$, and therefore $int(X)\subseteq X_{uni}$. This completes the proof when $d=1$.

Now let us assume our result holds for all $d< d^*$. We now show our result is true when $\Phi$ acts upon $\mathbb{R}^{d^*}$. Fix $\x\in int(X)$. Our strategy of proof will be to show that there exists $\a\in \D^*$ such that $T_{\a}(\x)\in int(\textrm{conv}(B))$, where $B\subset F$ consists of $d^*+1$ fixed points not contained in a $(d^*-1)$-dimensional affine subspace. Our result will then follow from Proposition \ref{universal prop}.

By Lemma \ref{strong caratheodory} there exists a set of $d^*+1$ extremal fixed points $B'\subseteq F$ such that $\x\in \textrm{conv}(B')$ and $B'$ is not contained in any $(d^*-1)$-dimensional affine subspace of $\mathbb{R}^{d^*}$. If  $\x\in int(\textrm{conv}(B'))$, then our result follows from Proposition \ref{universal prop}. Suppose not and assume $\x$ is contained in the boundary of $\textrm{conv}(B')$. In which case $\x$ is contained in the convex hull of $d^*$ elements from $B'.$ If $\x$ is in the interior of the convex hull of these $d^*$ elements we stop. Here the topology used to define the interior is that obtained by identifying the convex hull of these $d^*$ elements with a subset of $\mathbb{R}^{d^*-1}$. If $\x$ is not in the interior of the convex hull of these $d^*$ elements, then it must be contained in the convex hull of $d^*-1$ elements from $B'.$ If $\x$ is contained in the interior of the convex hull of these $d^*-1$ elements we stop. If not then $\x$ must be contained in the convex hull of $d^*-2$ elements from $B'$ and so on. Repeating this step must eventually yield $1\leq l\leq d^*-1$ such that $\x$ is contained in the interior of the convex hull of $l+1$ elements from $B'.$ For if not $\x$ would be in the convex hull of a single element of $B',$ and would therefore in fact equal an element of $B'$. This is not possible since each element of $B'$ is an extremal point of $X$ and $\x\in int(X)$. Summarising this argument, we may conclude that if $\x\notin int(\textrm{conv}(B'))$ there exists $1\leq l\leq d^*-1$ and $B_l'\subseteq B'$ such that $\#B_l'=l+1,$ $B_l'$ is not contained in any $(l-1)$-dimensional affine subspace,  and $\x\in int(\textrm{conv}(B_l')).$

The set $\textrm{conv}(B_l')$ is contained in a unique $l$-dimensional affine subspace of $\mathbb{R}^d$ that we denote by $W$. By elementary linear algebra, if $H$ is an $l'$-dimensional affine subspace where $l'\leq l$, it is the case that either $W=H,$ $W\cap H=\emptyset,$ or $W\cap H$ is an affine subspace of dimension strictly less than $l$. This means that if $A\subseteq F$ and $\dim(\textrm{conv}(A))\leq l,$ then one of the following options must hold: 
\[\textrm{conv}(A)\subseteq W,\quad \textrm{conv}(B_l')\cap\textrm{conv}(A)=\emptyset,\quad \textrm{or}\quad \dim(\textrm{conv}(B_l')\cap\textrm{conv}(A))<l.\]
 Here $\dim(Y)$ denotes the topological dimension of the smallest affine subspace containing $Y$ for $Y\subseteq \mathbb{R}^d$. It follows from these facts that if $\lambda$ is chosen to be sufficiently close to $1,$ in a way that depends only upon $F$, then there exists a compact subset $K$ contained in $int(\textrm{conv}(B_l')),$ a digit $i\in \D$, and $r>0$ such that the following properties hold:
\begin{enumerate}
	\item For all $\mathbf{y}\in K$ we have $$\Big(B(\mathbf{y},r)\setminus \textrm{conv}(B_l')\Big)\bigcap  \bigcup_{\substack{A\subseteq F\\ \dim(\textrm{conv}(A))\leq l}}  \textrm{conv}(A)=\emptyset.$$
	\item For all $\mathbf{y}\in K$ we have $$T_i(\mathbf{y})\in B(\mathbf{y},r)\setminus \textrm{conv}(B_l')$$  
	\item $int(K)\neq \emptyset.$
	\item For all $\y\in K$ we have $B(\mathbf{y},r)\subseteq int(X).$
\end{enumerate} 
Note that in item $(2)$ we can simply choose $i\in \D$ such that $\p_i\notin W$.

Since $\x\in int(\textrm{conv}(B_l'))$ and $l<d^*,$ we can apply our inductive hypothesis and Lemma \ref{dense} to assert that if $\lambda$ is sufficiently close to $1$ in a way that depends upon $B_l'$, then there exists a finite word $\a_0\in\cup_{j=0}^{\infty}\{i:\p_i\in B_l'\}^j$ such that $T_{\a_0}(\x)\in K$. Here we used the fact that $int(K)\neq \emptyset$. We then apply $T_{i}$ to $T_{\a_0}(\x)$, where $T_i$ is as in item $(2)$ above. It follows from items $(1)$ and $(4)$ that $T_i(T_{\a_0}(\x))\in int(X)$ and $T_i(T_{\a_0}(\x))\notin \textrm{conv}(A)$ for any $A\subseteq F$ such that $\dim(\textrm{conv}(A))\leq l$. We now apply Lemma \ref{strong caratheodory} again to assert that there exists a set of $d^*+1$ extremal  fixed points $B''\subseteq F$ such that $T_i(T_{\a_0}(\x))\in \textrm{conv}(B'')$ and $B''$ is not contained in a $(d^*-1)$-dimensional affine subspace. If $T_i(T_{\a_0}(\x))\in int( \textrm{conv}(B''))$ then we can apply Proposition \ref{universal prop} to complete our proof. If not, then $T_i(T_{\a_0}(\x))$ is contained in the boundary $\textrm{conv}(B'')$. Since $T_i(T_{\a_0}(\x))\notin \textrm{conv}(A)$ for any $A\subseteq F$ such that $\dim(\textrm{conv}(A))\leq l$, if $T_i(T_{\a_0}(\x))$ is contained in the boundary of $ \textrm{conv}(B'')$ and we repeat the argument given at the start of this proof, this argument will yield $l_1\geq l+1$ and $B_{l_1}''\subset B''$ such that $\#B_{l_1}''=l_1+1$, $B_{l_1}''$ is not contained in any $(l_1-1)$-dimensional affine subspace, and $T_i(T_{\a_0}(\x))\in int(\textrm{conv}(B_{l_1}'')).$ Otherwise we would have $T_i(T_{\a_0}(\x))\in \textrm{conv}(A)$ for some $A$ with $\dim(\textrm{conv}(A))\leq l$.

Repeating our previous arguments we can define a new compact subset $K$ contained in $int(\textrm{conv}(B_{l_1}'')),$ a digit $i\in \D,$ and $r>0$ such that properties analogous to $(1), (2),$ $(3)$ and $(4)$ hold for the set $B_{l_1}''$ when $\la$ is sufficiently close to $1$ in a way that depends only upon $F$. By an analogous argument to that following the statement of these properties, it follows that $T_i(T_{\a_0}(\x))$ can either be mapped into the interior of the convex hull of $d^*+1$ extremal fixed points that are not contained in any $(d^*-1)$-dimensional affine subspace, or $T_i(T_{\a_0}(\x))$ can be mapped into the interior of the convex hull of at least $l_2+1$ extremal fixed points that are not contained in any $(l_2-1)$-dimensional affine subspace, where $l_2\geq l_1+1$. In the first case we can apply Proposition \ref{universal prop} to complete our proof. If we are in the latter case and $T_i(T_{\a_0}(\x))$ has been mapped into the interior of the convex hull of $l_2+1$ extremal fixed points, we may again repeat the above step and define new analogues of $K$, $i$, and $r$. 

These steps cannot be repeated indefinitely. As such we may conclude that eventually either $\x$ is mapped into the interior of the convex hull of $d^*+1$ extremal fixed points that are not contained in any $(d^*-1)$-dimensional affine subspace, or $\x$ is mapped into the interior of the convex hull determined by $d^*$ extremal fixed points that are not contained in any $(d^*-2)$-dimensional affine subspace. In the former case we can apply Proposition \ref{universal prop} to complete our proof. In the latter case, repeating the above argument, we see that we can map this image of $\x$ outside of the convex hull of these $d^*$ fixed points in such a way that it is mapped into $int(X),$ and this new image of $\x$ is not contained in $\textrm{conv}(A)$ for any $A\subseteq F$ with $\dim(\textrm{conv}(A))\leq d^*-1$. Applying Lemma \ref{strong caratheodory} we see that $\x$ must have been mapped into the interior of the convex hull determined of $d^*+1$ extremal fixed points that are not contained in a $(d^*-1)$-dimensional affine subspace. In which case we can apply Proposition \ref{universal prop}. This completes our proof. 
\end{proof}

Theorem \ref{k normal theorem} now follows almost immediately from Theorem \ref{universal theorem} and Proposition \ref{prop1}.

\begin{proof}[Proof of Theorem \ref{k normal theorem}]
By the remarks preceding the proof of Proposition \ref{d+1 prop} it suffices to show that $int(X)\subset X_k$. Let $\delta_k:=\delta(k,F)>0$ be such that if $\lambda\in(1-\delta_k,1)$ then $X_{uni}=int(X)$ and $X_{k}$ contains an open dense subset of $X$. Such a $\delta_k$ exists by Theorem \ref{universal theorem} and Proposition \ref{prop1}. Let us call this open dense subset $O$. Fix $\x\in int(X).$ Then $\x$ has a universal coding. By Lemma \ref{dense} there exists $\a\in \D^{*}$ such that $T_{\a}(\x)\in O$. It follows from Lemma \ref{bijection lemma} and the fact that whether a sequence is $k$-simply normal does not depend on the initial block that $\x\in X_k$. Since $\x$ was arbitrary, this completes our proof. 
\end{proof}

\section{Final discussion}
\label{final discussion} Theorem \ref{k normal theorem} asserts that for any $F$ and $k\in\mathbb N$ there exists $\delta_k>0$ depending upon $F$ and $k$ such that if $\lambda\in(1-\delta_k,1),$ then $X_k=int(X)$. Similarly, Theorem \ref{universal theorem} asserts that for any $F$ there exists $\delta_{uni}>0$ depending upon $F$ such that if $\lambda\in(1-\delta_{uni},1),$ then $X_{uni}=int(X).$ We expect that one can reduce this dependence and conjecture that the following statements are true:
\begin{itemize}
	\item There exists $\delta'>0$ depending only upon $k\in\mathbb N$ and the dimension of the Euclidean space $\Phi$ acts upon such that if $\lambda\in(1-\delta',1),$ then $X_k=int(X)$.
	\item There exists $\delta''>0$ depending only upon the dimension of the Euclidean space $\Phi$ acts upon such that if $\lambda\in(1-\delta'',1),$ then $X_{uni}=int(X)$.
\end{itemize}Unfortunately, due to the delicate geometric arguments used in the proof of Theorem \ref{universal theorem} and the non effectiveness of Proposition \ref{prop1}, we are currently unable to provide a solution to either of these conjectures. Fortunately we can prove that both of these statements hold when $d=1$. 
\begin{theorem}
	\label{explicit universal thm}
	Assume $d=1$ and $\Phi$ is homogeneous. If $\lambda\in(2^{-1/2},1)$ and $\lambda^{-2}$ is not Pisot, then $int(X)=X_{uni}$. In particular if $\lambda\in(1.3247^{-1/2},1)$ then $int(X)=X_{uni}$.
\end{theorem} Theorem \ref{explicit universal thm} is a consequence of Theorem \ref{Feng thm}, Theorem \ref{EK}, and Proposition \ref{universal prop}. We leave the details to the interested reader.

\begin{theorem}
	\label{explicit thm}
	Assume $d=1,$ $k\geq 1$ and $\Phi$ is homogeneous. Then for any $\lambda \in (\max\{(\frac{1}{2})^{\frac{1}{k\cdot (n+1)^k}},1.3247^{-1/2}\},1)$ we have $X_{k}=int(X)$.
\end{theorem}

\begin{proof}
Write $\la_k:=\max\{(\frac{1}{2})^{\frac{1}{k\cdot (n+1)^k}},1.3247^{-1/2}\}$. By Theorem \ref{d=1 theorem} we know that for any $\lambda\in (\la_k,1)$ the set $X_k$ contains an open dense subset. By Theorem \ref{explicit universal thm} we know that for $\lambda\in (\la_k,1)$ we have $int(X)=X_{uni}$. Making use of Lemma \ref{dense} we can now argue as in the proof of Theorem \ref{k normal theorem} to show that $X_k=int(X)$.
\end{proof}

We can extend Theorem \ref{explicit universal thm} under an additional assumption to higher dimensions. The following theorem is an immediate corollary of Proposition \ref{universal prop}.

\begin{theorem}
	\label{explicit universal thm2}
	Assume $\Phi$ is homogeneous,  $\lambda\in(2^{-1/2d},1)$ and $\lambda^{-2d}$ is not a Pisot number. If every $\x\in int(X)$ is in the interior of $\textrm{conv}(B)$ for some $B$ consisting of $d+1$ fixed points that are not contained in any $(d-1)$-dimensional affine subspace, then $int(X)=X_{uni}$. 
\end{theorem}
We emphasise here that there are examples of $X$ such that there exists $\x\in int(X)$ and $\x$ is not in the interior of $\textrm{conv}(B)$ for any $B$ consisting of $d+1$ fixed points. Consider the case where $F=\{(0,0),(1,0),(0,1),(1,1)\},$ $X=[0,1]\times [0,1],$ and  $\x=(1/2,1/2).$ As an application of Theorem \ref{explicit universal thm2} we consider the following example.

\begin{example}
Let $\{\p_0,\ldots,\p_5\}$ be the vertices of a regular hexagon $X$. Then for any $\lambda\in(2^{-1/4},1)$ such that $\lambda^{-4}$ is not a Pisot number we have $int(X)=X_{uni}$. We can verify that $X$ satisfies the remaining hypothesis of Theorem \ref{explicit universal thm2} by inspection of Figure \ref{fig1}.

\begin{figure}[ht]
	\centering
	\begin{tikzpicture}[x=3.5,y=3.5]
	\path[draw][thick](0,0) -- (40,0) -- (60, 34.64) -- (40, 69.28) -- (0,69.28) -- (-20,34.64) -- (0,0);
	\path[draw][dashed](0,0)--(60,34.64);
	\path[draw][dashed](0,0)--(40,69.28);
	\path[draw][dashed](0,0)--(0,69.28);
	\path[draw][dashed](40,0)--(-20,34.64);
	\path[draw][dashed](40,0)--(0,69.28);
	\path[draw][dashed](40,0)--(40,69.28);
	\path[draw][dashed](60, 34.64)--(-20, 34.64);
	\path[draw][dashed](60, 34.64)--(0, 69.28);
	\path[draw][dashed](40,69.28)--(-20, 34.64);
	\end{tikzpicture}
	\caption{Each $\x\in int(X)$ is contained in $int(\textrm{conv}(B))$ for some $B$ consisting of three vertices of $X$. }
	\label{fig1}
\end{figure}
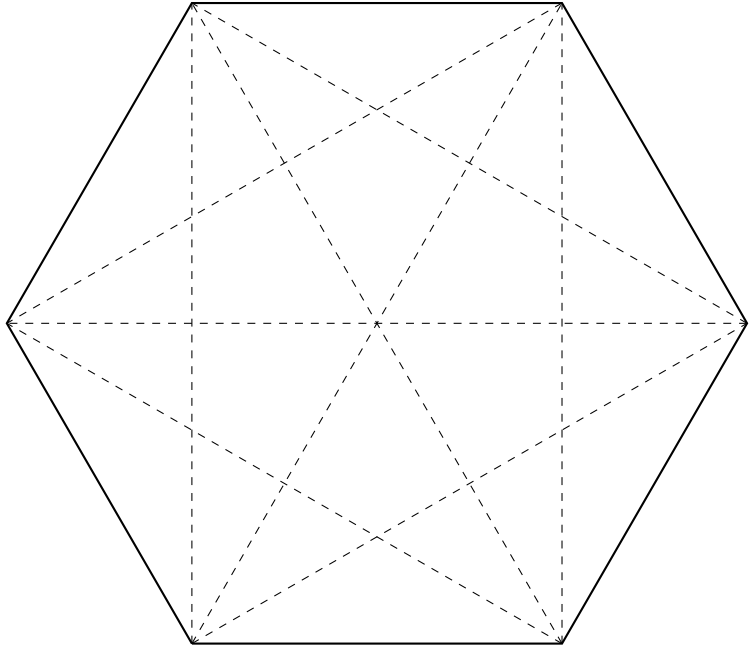
\end{example}

As an application of Theorem \ref{explicit thm} we consider the $q$-expansions studied by Erd\H{o}s and Komornik.

\begin{example}
	Let $q\in (1,2)$. Then for every $x\in [0,\frac{1}{q-1}]$ there exists  $\a\in\{0,1\}^{\mathbb{N}}$ such that $$x=\sum_{i=1}^{\infty}\frac{a_i}{q^i}.$$ Recall that such an $\a$ is called a $q$-expansion of $x$. A sequence $\a$ is a $q$-expansion of $x$ if and only if $\a$ is a coding for $x$ for the IFS $\{\frac{x}{q},\frac{x+1}{q}\}$. Theorem \ref{explicit thm} doesn't immediately apply to this IFS since for this family of IFSs the fixed points vary. However, by a straightforward scaling argument this issue can be overcome and one can prove that if $q\in(1, \min\{2^{\frac{1}{k\cdot 2^k}},1.3248^{1/2}\})$, then every $x\in(0,\frac{1}{q-1})$ has a $k$-normal $q$-expansion. We include a table of values for $\min\{2^{\frac{1}{k\cdot 2^k}},1.3248^{1/2}\}$ for $k\geq 2$ in Figure \ref{fig2}. The optimal parameter space of $q$ for which every $x\in(0,\frac{1}{q-1})$ has a $1$-normal $q$-expansion was determined in \cite{BakD,BakKong}. 
	
\begin{figure}
	\centering
	\begin{tabular}{ |c|c| } 
		\hline
		k & $\min\{2^{\frac{1}{k\cdot 2^k}}, 1.3248^{1/2}\}$  \\
		\hline
		2 & $ 1.0905\ldots$ \\ 
		3 & $1.0293\ldots$ \\ 
		4 & $1.0109\ldots$ \\ 
		5 & $1.0043\ldots$  \\
		6 & $1.0018\ldots$ \\
		7& $1.0008\ldots$ \\
		8& $1.0003\ldots$ \\
		9& $1.0001\ldots$ \\
		\hline
	\end{tabular}
	\caption{A table of values for $\min\{2^{\frac{1}{k\cdot 2^k}}, 1.3248^{1/2}\}$ .}
	\label{fig2}
\end{figure}	
	
\end{example}

It would be interesting to know how optimal the parameter space appearing in Theorem \ref{explicit thm} is. With that in mind we introduce the following, for each $k\geq 1$ and $n\geq 1$ let $$C(k,n):=\sup\{\delta:\textrm{ If } \lambda\in(1-\delta,1) \textrm{ then }X_k=int(X) \textrm{ for any }F\subseteq \mathbb{R} \textrm{ such that }\#F=n+1\}.$$ By Theorem \ref{explicit thm} we know that $C(k,n)\geq 1- \max\{(\frac{1}{2})^{\frac{1}{k\cdot (n+1)^k}},1.3247^{-1/2}\}$. Because of the $(n+1)^{-k}$ term appearing in the exponent of $1/2,$ the right hand side converges to zero very quickly (see Figure \ref{fig2}). It would be interesting to determine whether one could prove that $C(k,n)$ accumulates to zero at a significantly slower rate. More interesting still would be to determine whether in fact $C(k,n)$ decays to zero at all. This gives rise to the following conjectures which we state in arbitrary dimensions:

\begin{itemize}
	\item There exists $\delta_{nor}$ depending only on $d$ such that if $\lambda\in(1-\delta_{nor},1),$ then every $\x\in int(X)$ has a normal coding.
	\item For any $F$ there exists $\delta_{nor}:=\delta_{nor}(F)$ such that if $\lambda\in(1-\delta_{nor},1),$ then every $\x\in int(X)$ has a normal coding.
\end{itemize}
Recall that a coding is normal if it is $k$-simply normal for all $k$. Clearly the second conjecture is weaker than the first. We include it for completion.

It would also be interesting to construct a specific IFS for which every $\x\in int(X)$ had a normal coding. Progress with any of these problems seems well out of reach of our current methods. 

Theorems \ref{k normal theorem} and \ref{universal theorem} are phrased for homogeneous IFSs. One should expect that analogous results hold when our IFS has different rates of contraction. The main difficulty in proving such a result is proving an appropriate analogue of Proposition \ref{d+1 prop}. This proposition relies heavily on the fact the IFS is homogeneous. 

The results of this paper were phrased for IFSs where every similitude was of the form described by \eqref{similarity}. A general similitude can be expressed as $S=\lambda\cdot O+\mathbf{t},$ where $\lambda\in(0,1)$, $O$ is a $d\times d$ orthogonal matrix, and $\mathbf{t}\in\mathbb{R}^d.$ In our results $\Phi$ always consisted of similarities $\{S_i\}$ where the orthogonal matrix appearing in this decomposition was the identity. It would be interesting to extend the results of this paper to allow for non-trivial orthogonal matrices.

\section*{Acknowledgments}

The authors were supported by an LMS Scheme 4 grant. The first author was supported by EPSRC grant EP/M001903/1. The second author was supported by NSFC No. 11401516. He would like to thank the Mathematical Institute of Leiden University.

\end{document}